\documentclass{article}
\usepackage{amssymb}
\usepackage{amsmath,amsfonts,amsthm,amstext}

\usepackage{xcolor}
\newtheorem{theorem}{Theorem}[section]

\newtheorem{lemma}[theorem]{Lemma}
\newtheorem{proposition}[theorem]{Proposition}

\numberwithin{equation}{section}

\newcommand{\q}{\mathbb{Q}}
\newcommand{\z}{\mathbb{Z}}

\newcommand{\zpq}{\mathbb{Z}_p[[Q]]}

\begin{document}

\author{Dessislava H. Kochloukova\thanks{The first author was partially supported by grant CNPq 305457/2021-7 and FAPESP 2024/14914-9.} ~\&~
Aline G.S. Pinto\thanks{ The second author was partially supported by FAPDF.}\\
State University of Campinas (UNICAMP), \\  Campinas 13083-859, Brazil  \\
Department of Mathematics,~University of Bras\'\i lia,\\ 70910-900
Bras\'\i lia DF,~Brazil}

\title{Homological growth of nilpotent-by-abelian pro-p groups}

\date{}

\maketitle

\begin{abstract}
We show that the torsion-free rank of $H_i(M, \mathbb{Z}_p)$ has finite upper bound for $i \leq m$, where $M$ runs through the pro-$p$ subgroups of finite index in a pro-$p$ group $G$ that is (nilpotent of class $c$)-by-abelian  such that $ G/N'$ is of type $FP_{2cm}$.

\end{abstract}

\bigskip



\section{Introduction}
 In this paper we study how the rank $$ \mathrm{rk} ~ H_i(M,\mathbb{Z}_p) = dim_{\mathbb{Q}_p}  H_i(M, \mathbb{Z}_p) \otimes_{\mathbb{Z}_p} \mathbb{Q}_p$$ varies when $M$ runs through the set of all pro-$p$ subgroups of finite index in a fixed pro-$p$ group $G$ that is nilpotent-by-abelian. The case when $G$ is central-by-metabelian was previously resolved in \cite{KP-centre-by-metabelian} and was inspired by \cite{K-M} where the case of abstract groups that are abelian-by-polycyclic was considered.  The case when $G$ is  a  finitely presented pro-$p$ group, that is nilpotent-by-abelian and $i = 1$  was considered by Bridson and Kochloukova in \cite{B-K}. 
  In order the rank to have finite upper limit we need some strong homological conditions on the group $G$. We say that a pro-$p$ group $G$ is of type $FP_k$ if the trivial $\mathbb{Z}_p[[G]]$-module $\mathbb{Z}_p$ has a projective resolution with all projectives finitely generated in dimensions $ \leq k$. This is equivalent with all homology groups $H_i(G, \mathbb{F}_p)$ being finite for $ i \leq k$. Property $FP_2$ for pro-$p$ groups is equivalent with finite presentability. 

The growth of homologies in subgroups of finite index in pro-$p$ groups was earlier considered by Kochloukova and Zalesski in  \cite{K-Z2} for a special class $\mathcal{L}$ of pro-$p$ groups that are of type $FP_{\infty}$. A related problem is the study of the torsion of the abelianization of a pro-$p$ group i.e. the torsion of the first homology group. The growth of the torsion-rank of the first homology group was considered by Nikolov in \cite{Ni}.

 In general calculating homology or cohomology groups of a profinite group is not an easy task. There is better understanding in the case of special classes of groups, as $p$-adic analytic groups, as shown in \cite{S-W}. 
Furthermore little is known about finitely presented soluble pro-$p$ groups. In \cite{King} King described finitely presented metabelian pro-$p$ groups. This was later generalised by Kochloukova in \cite{Desi1} for metabelian groups of type $FP_{m}$ for $m \geq 3$. King's description of finite presentability used a specific invariant $\Delta$ that was later generalised by Kochloukova and Zalesskii in \cite{K-Z}.

 The following is our main result. It is motivated by the main result of \cite{Mirzaii-Mokari} where the case of abstract groups is considered. Our assumptions of the condition $FP_{2d}$ imposed  on the group $G/ N'$   are slightly different than the ones imposed in the abstract case in \cite{Mirzaii-Mokari} as  we know in the pro-$p$ case that the condition that $G/ N'$  is $FP_{2d}$ is equivalent to $\widehat{\otimes}_{\mathbb{Z}_p}^{2d} N/ N'$ is finitely generated as a $\mathbb{Z}_p [[Q]]$-module via the diagonal $Q$-action and we do not know whether a  pro-$p$ version  of a result of \cite{A} used in \cite{Mirzaii-Mokari} holds.
 
\begin{theorem}\label{principal}
Let $ 1\to N \to G \to Q \to 1$ be an short exact sequence of pro-$p$ groups, where $G$ is finitely generated, $N$ is nilpotent of class $c$ and $Q$ is abelian. Let $N'$ be the commutator subgroup of $N$ and suppose that the metabelian quotient $G/N'$ of $G$ is of type $FP_{2d}$, where $d = cm$. Then 
$$
\sup_{M\in \mathcal{A}} \mathrm{rk} ~ H_i(M,\mathbb{Z}_p) < \infty, ~~~~ \forall ~ 0\leq i \leq m ,
$$
where $\mathcal{A}$ is the set of all subgroups of $p$-power index in $G$ and, for an abelian pro-$p$ group $B$, rk$B:=\dim_{\mathbb{Q}_p} B\otimes_{\mathbb{Z}_p}\mathbb{Q}_p$ is the torsion-free rank of $B$.

\end{theorem}



\section{Preliminaries}

For preliminaries on homology of profinite groups we refer the reader to \cite{book}.


\begin{lemma}[{\cite[Lemma 4.1]{KP-centre-by-metabelian}}]\label{abelian}  Let $A$ be an abelian pro-$p$ group. Then

    \begin{itemize}
    
	\item[a)] $H_i(A, \mathbb{Z}_p) \otimes_{\mathbb{Z}_p} {\mathbb{Q}_p} \simeq (\widehat{\bigwedge}^i_{\mathbb{Z}_p} A)  \otimes_{\mathbb{Z}_p} {\mathbb{Q}_p}  \hbox{ for all } i \geq 1;
$

\item[b)] if $Q$ is a finitely generated pro-$p$ abelian group and $A$ a finitely generated, pro-$p$ $\mathbb{Z}_p[[Q]]$-module we have

$
H_i(Q, H_j(A, \mathbb{Z}_p)) \otimes_{\z_p} \q_p \simeq 
H_i(Q, \widehat{\bigwedge}_{\z_p}^j A) \otimes_{\z_p} \q_p \hbox{ for }  i \geq 0, j \geq 1.
$
    \end{itemize}
\end{lemma}

\medskip

\begin{lemma}[{\cite[Lemma 4.2]{KP-centre-by-metabelian}}] \label{finitesubgroup}
Let $G$ be a pro-$p$ group, $G_0$ a pro-$p$ open, normal, subgroup in $G$ and $V$ a pro-$p$ $\mathbb{Z}_p[[G]]$-module. Then
$$
H_n(G, V) \otimes_{\mathbb{Z}_p} {\mathbb{Q}_p}  \simeq H_0(G/G_0, H_n(G_0, V)) \otimes_{\mathbb{Z}_p} {\mathbb{Q}_p}. 
$$
\end{lemma} 

\medskip

\begin{lemma}[{\cite[Lemma 5.1]{KP-centre-by-metabelian}}] \label{2m to m} Let $Q$ be a finitely generated abelian pro-$p$ group and $B$ a finitely generated pro-$p$ $\zpq$-module such that $B\widehat{\otimes}_{\z_p}B$ is a finitely generated pro-$p$ $\zpq$-module via the diagonal $Q$-action. Then 
$$
\sup_{M\in\mathcal{A}} \dim_{\q_p} B \otimes_{\mathbb{Z}_p[[M]]}\q_p <\infty,
$$
where $\mathcal{A}$ is the set of all subgroups of $p$-power index in $G = B \rtimes Q$ and we view $B$ as $\mathbb{Z}_p[[G]]$-module via the canonical epimorphism $G \to Q$.
\end{lemma}

\begin{theorem}[{\cite[Theorem 5.5]{KP-centre-by-metabelian}}] \label{shift dimension}
Let $Q$ be a finitely generated abelian pro-$p$ group and $L$ a finitely generated pro-$p$ $\zpq$-module. If
$$
\sup_{t\geq 1} \dim_{\q_p} L\otimes_{\mathbb{Z}_p[[Q^{p^t}]]}\q_p < \infty,
$$
then
$$
\sup_{t\geq 1} \dim_{\q_p} H_i(Q^{p^t},L)\otimes_{\z_p}\q_p < \infty, ~~~~ \mbox{for all} ~ i.
$$ 
\end{theorem}

\noindent \textbf{Remark.} We will need later to apply Theorem \ref{shift dimension} in a more general situation when we do not know whether $L$ is a finitely generated  pro-$p$ $\zpq$-module but there is a 
finitely generated  pro-$p$ $\zpq$-submodule $L_0$ of $L$ such that the inclusion $L_0 \to L$ induces an isomorphism $L_0 \otimes_{\mathbb{Z}p} \mathbb{Q}_p \to L \otimes_{\mathbb{Z}_p} \mathbb{Q}_p$.
Indeed, in this case Theorem \ref{shift dimension} holds for $L$ substituted with $L_0$ and since we have the isomorphisms
$L_0\otimes_{\mathbb{Z}_p[[Q^{p^t}]]}\q_p \simeq L\otimes_{\mathbb{Z}_p[[Q^{p^t}]]}\q_p$ and  $H_i(Q^{p^t},L_0)\otimes_{\z_p}\q_p \simeq  H_i(Q^{p^t},L)\otimes_{\z_p}\q_p$, we conclude  the result holds for $L$.

\medskip \begin{lemma} \label{subquotients}
Let $Q$ be a finitely generated abelian pro-$p$ group and $V$ be a pro-$p$ $\mathbb{Z}_p[[Q]]$-module such that $\widehat{\bigotimes}^{n}_{\mathbb{Z}_p}V$ is a finitely generated pro-$p$ $\mathbb{Z}_p[[Q]]$-module  via the diagonal $Q$-action. If 
$$
\sup_{t\geq 1} \dim_{\mathbb{Q}_p} (\widehat{\bigotimes}^n_{\mathbb{Z}_p}V) \otimes_{\mathbb{Z}_p[[Q^{p^t}]]}\mathbb{Q}_p < \infty,
$$
then for any $\mathbb{Z}_p[[Q]]$-subquotient $U$ of $\widehat{\bigotimes}^n_{\mathbb{Z}_p}V$ we have
$$
\sup_{t\geq 1} \dim_{\mathbb{Q}_p} U \otimes_{\mathbb{Z}_p[[Q^{p^t}]]}\mathbb{Q}_p < \infty.
$$
    
\end{lemma}

\begin{proof}
Firstly, if $U = (\widehat{\bigotimes}_{\mathbb{Z}_p}^n V)/T$ is a $\mathbb{Z}_p[[Q]]$-quotient of $\widehat{\bigotimes}_{\mathbb{Z}_p}^nV$,  for some pro-$p$ $\mathbb{Z}_p[[Q]]$-submodule $T$ of $\widehat{\bigotimes}_{\mathbb{Z}_p}^n V$, then clearly
$$
\dim_{\mathbb{Q}_p} U\otimes_{\mathbb{Z}_p[[Q^{p^t}]]} \mathbb{Q}_p \leq \dim_{\mathbb{Q}_p} (\widehat{\bigotimes}^n_{\mathbb{Z}_p}V) \otimes_{\mathbb{Z}_p[[Q^{p^t}]]}\mathbb{Q}_p.
$$
From this, it follows immediately that 
$$
\sup_{t\geq 1}\dim_{\mathbb{Q}_p} U\otimes_{\mathbb{Z}_p[[Q^{p^t}]]} \mathbb{Q}_p \leq \sup_{t\geq 1}\dim_{\mathbb{Q}_p} (\widehat{\bigotimes}^n_{\mathbb{Z}_p}V) \otimes_{\mathbb{Z}_p[[Q^{p^t}]]}\mathbb{Q}_p < \infty.
$$

For the general case, let $U$ be a pro-$p$ $\mathbb{Z}_p[[Q]]$-submodule of some $W := (\widehat{\bigotimes}_{\mathbb{Z}_p}^n V)/T$. Then $W/U$ is of the form $(\widehat{\bigotimes}_{\mathbb{Z}_p}^n V)/T'$, for some pro-$p$ $\mathbb{Z}_p[[Q]]$-submodule $T'$ of $\widehat{\bigotimes}_{\mathbb{Z}_p}^n V$. So, as above,
$$
\sup_{t\geq 1}\dim_{\mathbb{Q}_p} W \otimes_{\mathbb{Z}_p[[Q^{p^t}]]} \mathbb{Q}_p < \infty ~~~ \mbox{and} ~~~ \sup_{t\geq 1}\dim_{\mathbb{Q}_p} (W/U) \otimes_{\mathbb{Z}_p[[Q^{p^t}]]}\mathbb{Q}_p < \infty.
$$
Since $\sup_{t\geq 1}\dim_{\mathbb{Q}_p} (W/U) \otimes_{\mathbb{Z}_p[[Q^{p^t}]]}\mathbb{Q}_p < \infty$, by Theorem \ref{shift dimension}, we have
$$
\sup_{t\geq 1}\dim_{\mathbb{Q}_p} H_i(Q^{p^t}, W/U) \otimes_{\mathbb{Z}_p} \mathbb{Q}_p < \infty,~~\mbox{for all}~ i. 
$$
Therefore, considering the long exact sequence for pro-$p$ homology associated to the short exact sequence of $\mathbb{Z}_p[[Q]]$-modules $0\to U \to W \to W/U \to 0$, 
$$
\cdots \to H_1(Q^{p^t}, W/U) \to U \widehat{\otimes}_{\mathbb{Z}_p[[Q^{p^t}]]}\mathbb{Z}_p \to  W \widehat{\otimes}_{\mathbb{Z}_p[[Q^{p^t}]]}\mathbb{Z}_p \to 
$$
$$
\hspace{8cm} \to W/U \widehat{\otimes}_{\mathbb{Z}_p[[Q^{p^t}]]}\mathbb{Z}_p \to 0,
$$
we conclude that  
$$
\sup_{t\geq 1}\dim_{\mathbb{Q}_p} U \otimes_{\mathbb{Z}_p[[Q^{p^t}]]}\mathbb{Q}_p ~ \leq ~
\sup_{t\geq 1}\dim_{\mathbb{Q}_p} W \otimes_{\mathbb{Z}_p[[Q^{p^t}]]} \mathbb{Q}_p 
~~ $$ $$
+ ~~ \sup_{t\geq 1}\dim_{\mathbb{Q}_p} H_1(Q^{p^t}, W/U) \otimes_{\mathbb{Z}_p} \mathbb{Q}_p 
< \infty.   \hspace{2.9cm}
$$
\end{proof}

\section{T-maps}

 This section aims to establish a relationship between the homologies of a nilpotent pro-p group and those of its finite index pro-p subgroups, even in the case where the nilpotent pro-p group is not finitely generated.

Let $$\mu : V \to W$$ be a homomorphism of pro-$p$ $\mathbb{Z}_p$-modules. We say that $\mu$ is a $T$-map if every element $v \in Ker(\mu)$ and $w \in Coker(\mu)$ is $\mathbb{Z}_p$-torsion i.e. there are $\lambda_1, \lambda_2 \in \mathbb{Z}_p \setminus \{ 0 \}$ such that $ \lambda_1 v = 0, \lambda_2 w = 0$. We also say $Ker(\mu)$ and $Coker(\mu)$ are $\mathbb{Z}_p$-torsion. 

\begin{lemma} \label{criterion}  Let $\mu : V \to W$ be a homomorphism of pro-$p$ $\mathbb{Z}_p$-modules. Then $\mu$ is a $T$-map if and only if $\mu$ induces an isomorphism $\mu_0 : V \otimes_{\mathbb{Z}_p} \mathbb{Q}_p \to W \otimes_{\mathbb{Z}_p} \mathbb{Q}_p$.
\end{lemma}

\begin{proof}
Consider the short exact sequences $$0 \to Ker(\mu) \to V \to Im (\mu) \to 0$$ and $$0 \to Im (\mu) \to W \to Coker(\mu) \to 0$$ Since $\otimes_{\mathbb{Z}_p} \mathbb{Q}_p$ is an exact functor we get exact sequences $$0 \to Ker(\mu) \otimes_{\mathbb{Z}_p} \mathbb{Q}_p \to V \otimes_{\mathbb{Z}_p} \mathbb{Q}_p \to Im (\mu)\otimes_{\mathbb{Z}_p} \mathbb{Q}_p \to 0$$
and $$0 \to Im (\mu) \otimes_{\mathbb{Z}_p} \mathbb{Q}_p \to W \otimes_{\mathbb{Z}_p} \mathbb{Q}_p \to Coker(\mu) \otimes_{\mathbb{Z}_p} \mathbb{Q}_p \to 0$$
Thus $$Ker(\mu_0) \simeq Ker(\mu) \otimes_{\mathbb{Z}_p} \mathbb{Q}_p \hbox{ and }  Coker(\mu_0) \simeq Coker(\mu) \otimes_{\mathbb{Z}_p} \mathbb{Q}_p$$
Finally for a $\mathbb{Z}_p$-module $U$ we have that $U$ is $\mathbb{Z}_p$-torsion if and only if $U \otimes_{\mathbb{Z}_p} \mathbb{Q}_p = 0$.
\end{proof}

\begin{lemma} \label{prod} If  $\mu_1 : V_1 \to W_1$ and $\mu_2 : V_2 \to W_2$ are $T$-maps then
$\mu_1 \widehat{\otimes}_{\mathbb{Z}_p} \mu_2 : V_1  \widehat{\otimes}_{\mathbb{Z}_p} V_2 \to W_1  \widehat{\otimes}_{\mathbb{Z}_p} W_2$ is a $T$-map.
\end{lemma}

\begin{proof}
Note that $$\mu_1 \widehat{\otimes}_{\mathbb{Z}_p} \mu_2  = (id_{W_1}  \widehat{\otimes}_{\mathbb{Z}_p} \mu_2 ) (\mu_1 \widehat{\otimes}_{\mathbb{Z}_p} id_{V_2})$$ and composition of $T$-maps is a $T$-map, hence it suffices to consider two special cases: $\mu_1$ or $\mu_2$ are the identity maps. Both cases are similar, hence without loss of generality we can assume that $V_2 = W_2$ and $\mu_2$ is the identity map. Then using that $ \widehat{\otimes}_{\mathbb{Z}_p}$ is right exact functor we have for 
$$\mu_1 \widehat{\otimes}_{\mathbb{Z}_p} id_{V_2} :  V_1  \widehat{\otimes}_{\mathbb{Z}_p} V_2 \to W_1  \widehat{\otimes}_{\mathbb{Z}_p} V_2$$
that $Coker(\mu_1 \widehat{\otimes}_{\mathbb{Z}_p} id_{V_2}) \simeq (Coker \mu_1) \widehat{\otimes}_{\mathbb{Z}_p}V_2$ is $\mathbb{Z}_p$-torsion, since  $Coker (\mu_1)$ is 
$\mathbb{Z}_p$-torsion. Furthermore, $Ker(\mu_1 \widehat{\otimes}_{\mathbb{Z}_p} id_{V_2} )$ is the image of $Ker(\mu_1) \widehat{\otimes}_{\mathbb{Z}_p} V_2$ in $ V_1  \widehat{\otimes}_{\mathbb{Z}_p} V_2$ and $Ker(\mu_1) \widehat{\otimes}_{\mathbb{Z}_p} V_2$ is $\mathbb{Z}_p$-torsion, since $Ker(\mu_1)$ is $\mathbb{Z}_p$-torsion. Thus, we conclude that $Ker(\mu_1 \widehat{\otimes}_{\mathbb{Z}_p} id_{V_2} )$ is $\mathbb{Z}_p$-torsion.

\end{proof}

\begin{lemma} \label{hom-chain}
Let $\mathcal{A}_1, \mathcal{A}_2$ be complexes of  pro-$p$ $\mathbb{Z}_p$-modules and $\theta: \mathcal{A}_1 \to \mathcal{A}_2$ be a chain map that at each dimension is a $T$-map, i.e. we say it is a chain $T$-map. Then, for every $i$, we have that the induced map
$$\theta_i : H_i ( \mathcal{A}_1) \to H_i(\mathcal{A}_2)$$
is a $T$-map.
\end{lemma}

\begin{proof}
Consider the complexes $\mathcal{A}_0 = Ker (\theta)$ and $\mathcal{A}_3 = Coker(\theta)$.

Consider the short exact sequence of complexes $0 \to \mathcal{A}_0 \to \mathcal{A}_1 \to Im (\theta) \to 0$ induced by $\theta$. Note that each module of the complex $\mathcal{A}_0$ is $\mathbb{Z}_p$-torsion, then the long exact sequence in homology implies that $$H_i(\mathcal{A}_1) \to H_i(Im (\theta))$$ is a  $T$-map.

Similarly consider the short exact sequence of complexes $0 \to Im (\theta) \to \mathcal{A}_2 \to  \mathcal{A}_3 \to 0$ induced by the embedding of $Im(\theta)$ in $\mathcal{A}_2 $. Using that each module of the complex $\mathcal{A}_3$ is $\mathbb{Z}_p$-torsion and the long exact sequence in homology implies that 
 $$H_i( Im (\theta)) \to H_i(\mathcal{A}_2)$$ is a $T$-map. Finally use that composition of $T$-maps is a chain $T$-map.
\end{proof}

\begin{lemma} \label{T-map abelian}Let $Z$ be an abelian pro-$p$ group, i.e. a pro-$p$ $\mathbb{Z}_p$-module, and $Z_1$ a pro-$p$ subgroup of $Z$ of finite index. Then the canonical map
$$H_n(Z_1, \mathbb{Z}_p) \to H_n(Z, \mathbb{Z}_p)$$
is a $T$-map.
\end{lemma}

\begin{proof}
By  Lemma \ref{finitesubgroup}, 
$H_n(Z, \mathbb{Z}_p) \otimes_{\mathbb{Z}_p} \mathbb{Q}_p \simeq H_0(Z/ Z_1, H_n(Z_1, \mathbb{Z}_p)) \otimes_{\mathbb{Z}_p} \mathbb{Q}_p$.  By the proof (it uses spectral sequence), this isomorphism is induced by the embedding of $Z_1$ in $Z$. But since $Z$ is abelian $H_0(Z/ Z_1, H_n(Z_1, \mathbb{Z}_p)) \simeq H_n(Z, \mathbb{Z}_p)$, hence the embedding of $Z_1$ in $Z$ induces an isomorphism 
 $$  H_n(Z_1, \mathbb{Z}_p) \otimes_{\mathbb{Z}_p} \mathbb{Q}_p \simeq H_n(Z, \mathbb{Z}_p) \otimes_{\mathbb{Z}_p} \mathbb{Q}_p$$
\noindent
Then we apply Lemma \ref{criterion}.

\end{proof}

\begin{proposition} \label{prop-Tmap}
Let $N$ be a pro-$p$ nilpotent group (not necessarily finitely generated) and $N_1$ be a pro-$p$ subgroup of finite index. Then, for each $i$, we have that the canonical map
$$H_i(N_1, \mathbb{Z}_p) \to H_i(N, \mathbb{Z}_p)$$
is a $T$-map.
\end{proposition}

\begin{proof} 
We induct on the nilpotency class of $N$.  The first step is done in Lemma \ref{T-map abelian}. Let $Z$ be the center of $N$ and $Z_1 = Z \cap N_1$. Consider the spectral sequence
$$E^2_{i,j} = H_i(N/Z, H_j(Z, \mathbb{Z}_p)) \simeq   H_i(N/Z, \mathbb{Z}_p) \otimes_{\mathbb{Z}
_p} H_j(Z, \mathbb{Z}_p)$$ that converges to $H_{n}(N, \mathbb{Z}_p)$
and the spectral sequence
$$\widetilde{E}^2_{i,j} = H_i(N_1/Z_1, H_j(Z_1, \mathbb{Z}_p)) \simeq   H_i(N_1/Z_1, \mathbb{Z}_p) \otimes_{\mathbb{Z}
_p} H_j(Z_1, \mathbb{Z}_p)$$ that converges to $H_{n}(N_1, \mathbb{Z}_p)$.

By induction, the canonical map $H_i(N_1/ Z_1, \mathbb{Z}_p) \to H_i(N/ Z, \mathbb{Z}_p)$ is a $T$-map and since $Z_1$ and $Z$ are abelian we have that $H_i(Z_1, \mathbb{Z}_p) \to H_i(Z, \mathbb{Z}_p)$ is a $T$-map. Then, by Lemma \ref{prod}, we have that 
the canonical map $\widetilde{E}^2_{i,j} \to E^2_{i,j}$ is a $T$-map and, by taking iterating homologies and by Lemma \ref{hom-chain}, we get that the canonical map $\widetilde{E}^{\infty}_{i,j} \to E^{\infty}_{i,j}$ is a $T$-map. Then by the convergence of the spectral sequences we get that $$H_i(N_1, \mathbb{Z}_p) \to H_i(N, \mathbb{Z}_p)$$
is a $T$-map. 
 \end{proof}

\section{Proof of Theorem \ref{principal}}

Let $1 \to N \to G \to Q \to 1$ be an short exact sequence of pro-$p$ groups, where $G$ is finitely generated, $N$ is nilpotent of class $c$ and $Q$ is abelian. Let $N'$ be the commutator subgroup of $N$ and suppose that the metabelian quotient $G/N'$ of $G$ is of type  $FP_{2d}$ where $d = mc$. We want to prove that  
$$
\sup_{M\in \mathcal{A}} \dim_{\mathbb{Q}_p} H_i(M,\mathbb{Z}_p)\otimes_{\mathbb{Z}_p}\mathbb{Q}_p < \infty, ~~~~ \forall ~ 0\leq i \leq m ,
$$
where $M$ runs through the class $\mathcal{A}$ of all finite index subgroups of $G$.

\medskip

Let $G_1$ be a pro-$p$ subgroup of finite index in $G$, let $Q_1$ be the image of $G_1$ in $Q$ and $N_1:= N \cap G_1$. Then $[Q:Q_1]<\infty$ and $[N:N_1]<\infty$. 
From the short exact sequence
$$ 1  \to N_1 \to G_1 \to Q_1 \to 1,$$
we obtain the Lyndon-Hochschild-Serre spectral sequence in pro-$p$ homology
\begin{equation}\label{LHS-1}   
E^2_{r,s}= H_r(Q_1,H_s(N_1,\mathbb{Z}_p)) \Longrightarrow H_{r+s}(G_1,\mathbb{Z}_p)\medskip
\end{equation}    
Then
\medskip \begin{equation}\label{soma1} 
\begin{array}{ll}
\dim_{\mathbb{Q}_p} H_j(G_1,\mathbb{Z}_p) \otimes_{\mathbb{Z}_p}\mathbb{Q}_p 
& = 
\sum^j_{r=0} \dim_{\mathbb{Q}_p} (E^{\infty}_{r,j-r} \otimes_{\mathbb{Z}_p}\mathbb{Q}_p ) \medskip \\
& \leq 
\sum^j_{r=0}\dim_{\mathbb{Q}_p} (E^{2}_{r,j-r} \otimes_{\mathbb{Z}_p}\mathbb{Q}_p).
\end{array}
\end{equation}\medskip

Since $[N:N_1]<\infty$, by Proposition \ref{prop-Tmap}, the map  $H_{j-r} (N_1, \mathbb{Z}_p) \to H_{j-r} (N, \mathbb{Z}_p)$, induced by the inclusion $N_1 \to N$, is a $T$-map. This implies
\begin{equation}\label{N-N1}
\dim_{\mathbb{Q}_p} E^{2}_{r,j-r} \otimes_{\mathbb{Z}_p}\mathbb{Q}_p = \dim_{\mathbb{Q}}H_r(Q_1,H_{j-r}(N,\mathbb{Z}_p)) \otimes_{\mathbb{Z}_p}\mathbb{Q}_p.
\end{equation}

\medskip Since $[Q:Q_1]< \infty$, there is $t>0$ such that $Q^{p^t}=\overline{\langle q^{p^t}~|~q\in Q\rangle} \subset Q_1$ and so, by Lemma \ref{finitesubgroup}, 
$$
H_r(Q_1,L)\otimes_{\mathbb{Z}_p}\mathbb{Q}_p ~\cong~ H_0(Q_1/Q^{p^t}, H_r(Q^{p^t},L)) \otimes_{\mathbb{Z}_p}\mathbb{Q}_p \cong H_r(Q^{p^t},L)_{Q_1/Q^{p^t}} \otimes_{\mathbb{Z}_p}\mathbb{Q}_p ,
$$
for any pro-$p$ $\mathbb{Z}_p[[Q_1]]$-module L. Hence
$$\dim_{\mathbb{Q}_p} H_r(Q_1,L)\otimes_{\mathbb{Z}_p}\mathbb{Q}_p \leq \dim_{\mathbb{Q}_p} H_r(Q^{p^t},L) \otimes_{\mathbb{Z}_p}\mathbb{Q}_p.$$
So, applying for $L=H_{j-r}(N,\mathbb{Z}_p)$, we get
$$\dim_{\mathbb{Q}_p} H_r(Q_1,H_{j-r}(N,\mathbb{Z}_p))\otimes_{\mathbb{Z}_p}\mathbb{Q}_p \leq \dim_{\mathbb{Q}_p} H_r(Q^{p^t},H_{j-r}(N,\mathbb{Z}_p)) \otimes_{\mathbb{Z}_p}\mathbb{Q}_p.
$$

\medskip Thus, from (\ref{soma1}), (\ref{N-N1}) and above, one gets
$$
\begin{array}{ll}
\dim_{\mathbb{Q}_p} H_j(G_1,\mathbb{Z}_p) \otimes_{\mathbb{Z}_p}\mathbb{Q}_p 
& \leq 
\sum^j_{r=0}\dim_{\mathbb{Q}_p} (E^{2}_{r,j-r} \otimes_{\mathbb{Z}_p}\mathbb{Q}_p) \medskip \\
& = 
\sum_{r=0}^{j} \dim_{\mathbb{Q}_p} (H_r(Q_1,H_{j-r}(N,\mathbb{Z}_p))\otimes_{\mathbb{Z}_p}\mathbb{Q}_p )\medskip \\
& \leq 
\sum_{r=0}^{j} \dim_{\mathbb{Q}_p} (H_r(Q^{p^t},H_{j-r}(N,\mathbb{Z}_p)) \otimes_{\mathbb{Z}_p}\mathbb{Q}_p). 
\end{array}
$$
So, to prove that
$$
\sup_{G_1\in \mathcal{A}}\dim_{\mathbb{Q}_p} H_j(G_1,\mathbb{Z}_p) \otimes_{\mathbb{Z}_p}\mathbb{Q}_p 
< \infty, ~~~~ \mbox{for all}~~ 0\leq j \leq m,
$$
it is sufficient to prove that 
$$
\sup_{t\geq 1} \dim_{\mathbb{Q}_p} H_r(Q^{p^t}, H_{j-r}(N,\mathbb{Z}_p)) \otimes_{\mathbb{Z}_p}\mathbb{Q}_p 
< \infty, ~~~~ \mbox{for all}~~ 0\leq r\leq j \leq m.
$$
\medskip
Now, by hypothesis, $G/N'$ is of type $FP_{2d}$. Note that $G/N'$ is a metabelian pro-$p$ group where $N/N'$ is an abelian normal pro-$p$ subgroup with finitely generated abelian pro-$p$ quotient $G/N = Q$. Then, by {\cite[Theorem D]{Desi1}}, we have that $$ \widehat{\bigotimes}^s_{\mathbb{Z}_p}N/N' \hbox{ is a finitely generated pro-p }$$ \begin{equation} \label{imp}  \mathbb{Z}_p[[Q]]\hbox{-module via the diagonal Q-action, for }0\leq s \leq 2d \end{equation} Thus, by Lemma \ref{2m to m}, 
\begin{equation}\label{nilpotent-tensors}
\sup_{t\geq 1} \dim_{\mathbb{Q}_p} ( \widehat{\bigotimes}^s_{\mathbb{Z}_p} N/N' )\otimes_{\mathbb{Z}_p[[Q^{p^t}]]}\mathbb{Q}_p  < \infty, ~~ \mbox{for}~~ 0\leq s\leq d.
\end{equation}

 By the remark after Theorem \ref{shift dimension}  applied for  $L =  H_j(N,\mathbb{Z}_p)$, the existence of an appropriate $L_0$ given by the remark after  Theorem \ref{filtration} and (\ref{imp}), we obtain that  if 
\begin{equation}\label{nilpotent-homology}
\sup_{t\geq 1} \dim_{\mathbb{Q}_p} H_j(N,\mathbb{Z}_p) \otimes_{\mathbb{Z}_p[[Q^{p^t}]]}\mathbb{Q}_p < \infty, ~~~ \mbox{for} ~~ 0\leq j\leq m
\end{equation}
then
$$
\sup_{t\geq 1} \dim_{\mathbb{Q}_p} H_r(Q^{p^t}, H_{j}(N,\mathbb{Z}_p)) \otimes_{\mathbb{Z}_p}\mathbb{Q}_p 
< \infty, ~~~~ \mbox{for all}~~ r.
$$
Thus, proving (\ref{nilpotent-homology}) and Theorem \ref{filtration} suffices to establish the theorem. 
We will show  Theorem \ref{filtration}  in the next section and (\ref{nilpotent-homology}) in Section \ref{final1}.


\section{Homology of nilpotent pro-$p$ subgroups}

 For finitely generated abelian-by-abelian pro-p groups $G$, it is well known that if  $1\to A \to G \to Q\to 1$ is a short exact sequence of pro-$p$ groups such that $A$ and $Q$ are abelian pro-$p$ groups, then the homology groups $H_j(A,\mathbb{Z}_p)$ are finitely generated as a pro-$p$ $\mathbb{Z}_p[[Q]]$-module, for $1\leq j\leq k$, if and only if the completed tensor products $\widehat{\bigotimes}^j_{\mathbb{Z}_p} A$ are finitely generated as a diagonal $\mathbb{Z}_p[[Q]]$-module, for $1\leq j\leq k$. And also this is equivalent to the metabelian pro-$p$ group $G$ being of type $FP_k$ (see \cite{Desi1}).  

In this section we will see in Theorem \ref{filtration} that for nilpotent-by-abelian pro-$p$ groups, if  $ 1 \to N \to G \to Q\to 1$ is a short exact sequence of pro-$p$ groups such that $N$ is a nilpotent pro-$p$ group of class $c$ and $Q$ an abelian pro-$p$ group, then the homology groups $H_j(N,\mathbb{Z}_p)$ considered as pro-$p$ $\mathbb{Z}_p[[Q]]$-modules are related with the completed tensor products $\widehat{\bigotimes}^s_{\mathbb{Z}_p} N/N'$, for $0\leq s\leq cj$, considered as a $\mathbb{Z}_p[[Q]]$-module via the diagonal $Q$-action. The statement of Theorem \ref{filtration} was based on the result \cite[Corollary 2.6]{Mirzaii-Mokari} for abstract groups, however the proof for pro-$p$ groups here differs from that for abstract groups. 

From now on, assume that $ 1 \to N \to G \to Q\to 1$ is a short exact sequence of pro-$p$ groups such that $N$ is a nilpotent pro-$p$ group of class $c$ and $Q$ an abelian pro-$p$ group. Consider the lower central series of $N$,
$$
1=\gamma_{c+1}(N) \subset \gamma_c(N)\subset \cdots \subset \gamma_2(N)\subset\gamma _1(N)=N.
$$
From the short exact sequence
$$
0\to \gamma_c(N) \to N \to N/\gamma_c(N) \to 0,
$$
we obtain the Lyndon-Hochschild-Serre spectral sequence 
\begin{equation}
    \label{LHS-gamma_cN}
    E^2_{r,s}=H_r(N/\gamma_c(N),H_s(\gamma_c(N),\mathbb{Z}_p)) \Longrightarrow H_{r+s}(N,\mathbb{Z}_p).
\end{equation}

Since $\gamma_{c+1}(N)=[\gamma_c(N),N]=1$, we have $\gamma_c(N)\subset Z(N)$, hence
$$
  E^2_{r,s}=H_r(N/\gamma_c(N), \mathbb{Z}_p) \otimes_{\mathbb{Z}_p} H_s(\gamma_c(N),\mathbb{Z}_p))$$

Moreover, we have a natural action of $Q$ on $N$ that provides a natural action of $Q$ on $H_s(\gamma_c(N),\mathbb{Z}_p)$ and $H_r(N/\gamma_c(N),\mathbb{Z}_p)$. From this we obtain a natural action of $Q$ on the spectral sequence (\ref{LHS-gamma_cN}). This means that the groups $E^2_{r,s}$ are pro-$p$ $\mathbb{Z}_p[[Q]]$-modules and the differentials $d^2_{r,s}$ are homomorphisms of pro-$p$ $\mathbb{Z}_p[[Q]]$-modules. So, $H_n(N, \mathbb{Z}_p)$ has a filtration of pro-$p$  $\mathbb{Z}_p[[Q]]$-modules  where each quotient of the filtration is a subquotient of $H_r(N/\gamma_c(N)) \widehat{\otimes}_{\mathbb{Z}_p} H_s(\gamma_c(N),\mathbb{Z}_p)$ with $r + s = n$.Using induction on $c$ we conclude that $H_n(N, \mathbb{Z}_p)$ has a filtration of pro-$p$  $\mathbb{Z}_p[[Q]]$-modules  where each quotient of the filtration is a subquotient of $$H_{r_1}(N/\gamma_2(N), \mathbb{Z}_p) \widehat{\otimes}_{\mathbb{Z}_p} H_{r_2}(\gamma_2(N)/ \gamma_3(N),\mathbb{Z}_p)  \widehat{\otimes}_{\mathbb{Z}_p} \ldots  \widehat{\otimes}_{\mathbb{Z}_p}  H_{r_c}(\gamma_c(N),\mathbb{Z}_p)$$ with $r_1 +  \ldots + r_c = n$.

 Let $A_i = \gamma_i(N)/ \gamma_{i+1}(N)$, $C_i$ the $\mathbb{Z}_p$-torsion part of $A_i$ and $B_i = A_i / C_i$. Then refining the filtration of $N$ using as quotients all $B_i$  and $C_i$, we get that 
 $H_n(N, \mathbb{Z}_p)$ has a filtration of pro-$p$  $\mathbb{Z}_p[[Q]]$-modules  where each quotient of the filtration is a subquotient of $$H_{b_1}(B_1, \mathbb{Z}_p) \widehat{\otimes}_{\mathbb{Z}_p} H_{c_1}(C_1, \mathbb{Z}_p) \widehat{\otimes}_{\mathbb{Z}_p}  \ldots \widehat{\otimes}_{\mathbb{Z}_p} H_{b_c}(B_c, \mathbb{Z}_p)\widehat{\otimes}_{\mathbb{Z}_p} H_{c_c}(C_c, \mathbb{Z}_p) $$ with $b_1 + c_1 +  \ldots + b_c + c_c = n$.

Since each $ H_{c_i}(C_i, \mathbb{Z}_p)$ is $\mathbb{Z}_p$-torsion for $c_i \geq 1$, we get that $H_n(N, \mathbb{Z}_p) \otimes_{\mathbb{Z}_p} \mathbb{Q}_p$ has a filtration  where each quotient is a subquotient of $$(H_{b_1}(B_1, \mathbb{Z}_p) \widehat{\otimes}_{\mathbb{Z}_p}  \ldots \widehat{\otimes}_{\mathbb{Z}_p} H_{b_c}(B_c, \mathbb{Z}_p)) \otimes_{\mathbb{Z}_p} \mathbb{Q}_p  $$ with $b_1  +  \ldots + b_c  \leq n$.

Since each $B_i$ is $\mathbb{Z}_p$-torsion-free, we have that $H_{b_i}(B_i, \mathbb{Z}_p) \simeq \widehat{\wedge}_{\mathbb{Z}_p} ^{b_i} B_i$.  Note that,  since $A_1=N/N'$, each $A_i$ is a subquotient of $\widehat{\otimes}_{\mathbb{Z}_p} ^{i} A_1$, hence $B_i$ is a subquotient of  $\widehat{\otimes}_{\mathbb{Z}_p} ^{i} A_1$ and 
$\widehat{\wedge}_{\mathbb{Z}_p} ^{b_i} B_i$ is a subquotient of  $\widehat{\otimes}_{\mathbb{Z}_p}  ^{i b_i} A_1$
Then
$$
H_{b_1}(B_1, \mathbb{Z}_p) \widehat{\otimes}_{\mathbb{Z}_p}  \ldots \widehat{\otimes}_{\mathbb{Z}_p} H_{b_c}(B_c, \mathbb{Z}_p)$$
is a subquotient of $\widehat{\otimes}_{\mathbb{Z}_p}  ^{b_1 + 2 b_2 + \ldots + c b_c} A_1$. Hence 
 $H_n(N, \mathbb{Z}_p) \otimes_{\mathbb{Z}_p} \mathbb{Q}_p$ has a filtration  where each quotient is a subquotient of some  $(\widehat{\otimes}_{\mathbb{Z}_p}  ^{b_1 + 2 b_2 + \ldots + c b_c} A_1) \otimes_{\mathbb{Z}_p} \mathbb{Q}_p$ with $b_1  +  \ldots + b_c  \leq n$, and so $b_1 + 2 b_2 + \ldots + c b_c \leq c n$. Since spectral sequence is a natural construction, we deduce the following result where $n$ is substituted with $j$.

\begin{theorem} \label{filtration}
Let $  1 \to N\to G\to Q\to 1$ be an exact sequence of pro-$p$ groups, where $N$ is a nilpotent pro-$p$ group of class $c$ and $Q$ is an abelian pro-$p$ group. Then there exists a natural filtration of $\mathbb{Z}_p[[Q]]$-submodules of $H_j(N,\mathbb{Z}_p)$,
$$
0=E_0 \subseteq E_1 \subseteq \cdots \subseteq E_{l-1}\subseteq E_l=H_j(N,\mathbb{Z}_p),
$$
such that for any $0\leq k \leq l$, $(E_k/E_{k-1}) \otimes_{\mathbb{Z}_p}\mathbb{Q}_p$ is a natural subquotient from the set
$$
\left\{ (\widehat{\bigotimes}^s_{\mathbb{Z}_p}  N/N') \otimes_{\mathbb{Z}_p}\mathbb{Q}_p\right\}_{0\leq s\leq cj},
$$
where $\widehat{\bigotimes}^s_{\mathbb{Z}_p}  N/N'$ is considered as a $\mathbb{Z}_p[[Q]]$-module via the diagonal $Q$-action.
\end{theorem}

\noindent \textbf{Remark.} Suppose that in the above theorem  each $\mathbb{Z}_p[[Q]]$-module $\widehat{\bigotimes}^s_{\mathbb{Z}_p}  N/N'$ is finitely generated for $s \leq cj$.
Fix $k$ and corresponding $s$ such that $(E_k/ E_{k-1}) \otimes_{\mathbb{Z}_p}\mathbb{Q}_p$ is a subquotient of 
 $ (\widehat{\bigotimes}^s_{\mathbb{Z}_p}  N/N') \otimes_{\mathbb{Z}_p}\mathbb{Q}_p$ i.e.  $$(E_k/ E_{k-1}) \otimes_{\mathbb{Z}_p}\mathbb{Q}_p \simeq V_k \otimes_{\mathbb{Z}_p}\mathbb{Q}_p$$ where $V_k$ is a subquotient of $(\widehat{\bigotimes}^s_{\mathbb{Z}_p}  N/N')$, so $V_k$ is 
 a finitely generated $\mathbb{Z}_p[[Q]]$-module (note $ \otimes_{\mathbb{Z}_p}\mathbb{Q}_p$ is an exact functor, so commutes with subquotients).
 Then there is  a finitely generated pro-$p$ $\mathbb{Z}_p[[Q]]$-submodule $\overline{L}_k$ of $E_k/ E_{k-1}$ such that  the embedding of  $\overline{L}_k \to E_k/ E_{k-1}$ is a $T$-map, i.e. induces an isomorphism  \begin{equation} \label{novo1} \overline{L}_k  \otimes_{\mathbb{Z}_p}\mathbb{Q}_p \simeq (E_k/ E_{k-1})  \otimes_{\mathbb{Z}_p}\mathbb{Q}_p. \end{equation}
 Then  we can find a finitely generated $\mathbb{Z}_p[[Q]]$-submodule $L_k$ of $E_k \subseteq H_j(N,\mathbb{Z}_p)$ such that 
 $\overline{L}_k$  is the image of $L_k$ in $E_k/ E_{k-1}$.
 
  Define $L_0$ as the $\mathbb{Z}_p[[Q]]$-submodule of $H_j(N,\mathbb{Z}_p)$ generated by all $L_k$, where $k \leq l$. 
 Then $L_0$ is finitely generated as $\mathbb{Z}_p[[Q]]$-module and, since (\ref{novo1}) holds for all $k \leq l$,   the inclusion map $L_0 \to H_j(N,\mathbb{Z}_p)$ induces an isomorphism $$L_0  \otimes_{\mathbb{Z}_p}\mathbb{Q}_p \simeq H_j(N,\mathbb{Z}_p) \otimes_{\mathbb{Z}_p}\mathbb{Q}_p,$$ i.e. the inclusion map $L_0 \to H_j(N,\mathbb{Z}_p)$   is a $T$-map.

\section{Coming back to the proof of Theorem \ref{principal}} \label{final1}

As indicated in (\ref{nilpotent-homology}), to complete the proof of Theorem \ref{principal} we must demonstrate that
$$
\sup_{t\geq 1} \dim_{\mathbb{Q}_p} H_j(N,\mathbb{Z}_p) \otimes_{\mathbb{Z}_p[[Q^{p^t}]]} \mathbb{Q}_p < \infty, ~~~ \mbox{for}~~ 0\leq j\leq m . 
$$

By Theorem \ref{filtration}, $H_j(N,\mathbb{Z}_p)$ has a natural filtration of $\mathbb{Z}_p[[Q]]$-modules
$$
0=E_0\subseteq E_1 \subseteq \cdots \subseteq E_{l-1}\subseteq E_l=H_j(N,\mathbb{Z}_p)
$$
such that, for any $0\leq k\leq l$, $(E_k/E_{k-1})\otimes_{\mathbb{Z}_p}\mathbb{Q}_p$ is a natural subquotient from the set 
$$
\left\{ (\widehat{\bigotimes}^s_{\mathbb{Z}_p} N/N') \otimes_{\mathbb{Z}_p} \mathbb{Q}_p \right\}_{0\leq s\leq cj},
$$
where $\widehat{\bigotimes}^s_{\mathbb{Z}_p} N/N'$ is considered as a $\mathbb{Z}_p[[Q]]$-module via the diagonal $Q$-action.

Now, remember we have concluded in (\ref{nilpotent-tensors}) that, from the hypothesis which $G/N'$ is of type $FP_{2d}$, where $d=cm$, by applying \cite[Theorem D]{Desi1} and Lemma \ref{2m to m}, we obtain 
$$
\sup_{t\geq 1} \dim_{\mathbb{Q}_p} (\widehat{\bigotimes}^s_{\mathbb{Z}_p} N/N') \otimes_{\mathbb{Z}_p[[Q^{p^t}]]} \mathbb{Q}_p < \infty, ~~~\mbox{for}~~ 0\leq s\leq d.
$$
Here it is worthwhile to emphasize that, comparing indexes, we need to consider $j \leq m$.

Again, since $Q$ is a finitely generated abelian pro-$p$ group and $\widehat{\bigotimes}^s_{\mathbb{Z}_p} N/N'$ is a finitely generated $\mathbb{Z}_p[[Q]]$-module via the diagonal $Q$-action, for $0\leq s \leq d$ (\cite[Theorem D]{Desi1}), by Lemma \ref{subquotients}, we obtain 
$$
\sup_{t\geq 1} \dim_{\mathbb{Q}_p} (E_k/E_{k-1}) \otimes_{\mathbb{Z}_p[[Q^{p^t}]]} \mathbb{Q}_p < \infty. 
$$
By induction on $k$, it follows that for any $1\leq k \leq l$
$$
\sup_{t\geq 1} \dim_{\mathbb{Q}_p} E_k \otimes_{\mathbb{Z}_p[[Q^{p^t}]]}\mathbb{Q}_p < \infty.
$$
Therefore 
$$
\sup_{t\geq 1} \dim_{\mathbb{Q}_p} H_j(N,\mathbb{Z}_p) \otimes_{\mathbb{Z}_p[[Q^{p^t}]]}\mathbb{Q}_p = \sup_{t\geq 1} \dim_{\mathbb{Q}_p} E_l \otimes_{\mathbb{Z}_p[[Q^{p^t}]]}\mathbb{Q}_p < \infty,
$$
for $0\leq j\leq m = d/ c$, as we wanted to prove.

  \begin{flushright} $\square$ \end{flushright}


\end{document}